\documentclass[a4paper,12pt]{amsart}
\usepackage[T1]{fontenc} 
\usepackage[utf8]{inputenc}
\usepackage[pdftex]{graphicx,color}
\usepackage{amssymb}
\usepackage{enumerate}
\usepackage{tikz}
\def\RR{\mathbb R}

\newcommand{\set}[1]{\left\{#1\right\}}

\newcommand{\qtq}[1]{\quad\text{#1}\quad}
\newcommand{\remove}[1]{ }

\newtheorem{theorem}{Theorem}[section]
\newtheorem{proposition}[theorem]{Proposition}
\newtheorem{lemma}[theorem]{Lemma}

\newtheorem*{theorem*}{Theorem}

\theoremstyle{definition}

\theoremstyle{remark}
\newtheorem*{remark}{Remark}
\newtheorem*{remarks}{Remarks}
\newtheorem*{example}{Example}

\numberwithin{equation}{section}
\numberwithin{figure}{section}

\begin{document}
\title{Critical bases for ternary alphabets}
\thanks{Version of 2016-04-05-b}
\author{Vilmos Komornik}
\address{16 rue de Copenhague\\
         67000 Strasbourg, France}
\email{vilmos.komornik@gmail.com}
\remove{\address{Département de mathématique\\
         Université de Strasbourg\\
         7 rue René Descartes\\
         67084 Strasbourg Cedex, France}
\email{komornik@math.unistra.fr}
}
\author{Marco Pedicini}
\address{Department of Mathematics and Physics\\
Roma Tre University\\
Lar\-go San Leonardo Murialdo 1\\
00146 Roma, Italy}      
\email{marco.pedicini@uniroma3.it}

\subjclass[2000]{Primary: 11A63, Secondary: 11B83}
\keywords{Ternary alphabet, unique expansion, non-integer base, beta-expansion, univoque sequence, Komornik--Loreti constant}
\thanks{Part of this work has been done during the visit of the first author at the Department of Basic and Applied Sciences for Engineering of the Sapienza University of Rome in 2014 and 2015, and at the Department of Mathematics of the Tor Vergata University of Rome in 2016.
The author thanks these institutions for their hospitality.}


\begin{abstract}
Glendinning and Sidorov discovered an important feature of the Komornik--Loreti constant $q'\approx1.78723$ in non-integer base expansions on two-letter alphabets: in bases $1<q<q'$ only countably numbers have unique expansions, while for $q\ge q'$ there is a continuum of such numbers. 
We investigate the analogous question for ternary alphabets.
\end{abstract}

\maketitle

\remove{Authors: 

\begin{itemize}
\item Vilmos Komornik, Département de mathématique,
         Université de Strasbourg,
         7 rue René Descartes,
         67084 Strasbourg Cedex, France,
         e-mail: \texttt{komornik@math.unistra.fr}
\item Marco Pedicini, Department of Mathematics and Physics,
Roma Tre University,
Lar\-go San Leonardo Murialdo 1,
00146 Roma, Italy, 
e-mail: \texttt{marco.pedicini@uniroma3.it}
\end{itemize}
}
\section{Introduction}\label{s1}

Given a real \emph{base} $q>1$ and a finite \emph{alphabet} $A\subset\RR$ (having at least two elements), by an \emph{expansion} of a real number $x$ we mean a sequence $c=(c_i)\subset A$ satisfying the equality
\begin{equation*}
\pi_q(c)=\pi_q(c_i):=\sum_{i=1}^{\infty}\frac{c_i}{q^i}=x.
\end{equation*}
We denote by $A^{\infty}$ the set of all sequences $c=(c_i)\subset A$, by $U_{A,q}$ the set of numbers $x$ having a unique expansion, and by $U'_{A,q}\subset A^{\infty}$ the set of the corresponding expansions.

The structure of the \emph{univoque set} $U_{A,q}$ is well known for the regular alphabets $A=\set{0,1,\ldots,m}$, $m=1,2,\ldots ;$ see, e.g., the reviews \cite{DevKom2016}, \cite{Kom2011} and their references.
For the general case a number of basic results have been given in \cite{Ped2005}.

Writing $a:=\min A$ and $b:=\max A$, it is clear that the lexicographically smallest and greatest sequences $a^{\infty}$ and $b^{\infty}$  belong to $U'_{A,q}$ for all $q>1$. 

Here and in the sequel we emply the notation of symbolic dynamics.
For example, we denote by $a^{\infty}$ the constant sequence $a,a,\ldots,$ by $(ab)^{\infty}$ the periodic sequence $a,b,a,b,\ldots,$ by $\set{ab,abb}^{\infty}$ the set of sequences formed by the blocks $A_1,A_2,\ldots,$ where each block is equal to one of the words $ab$ or $abb$, and by $d^*\set{ab,abb}^{\infty}$ the  union of the sequences of the form $c$, $dc$, $ddc$\ldots with some sequence $c\in\set{ab,abb}^{\infty}$.

For the regular alphabets $A=\set{0,1,\ldots,m}$ the univoque sets $U'_{A,q}$ are increasing with $q$.
In the general case the following holds:

\begin{proposition}\label{p11}
Fix an alphabet $A$.
\begin{enumerate}[\upshape (i)]
\item If $q>1$ is sufficiently close to one, then $U'_{A,q}=\set{a^{\infty},b^{\infty}}$.
\item If 
\begin{equation*}
q>1+\frac{a_J-a_1}{\min_{j>1} \set{a_j-a_{j-1}}},
\end{equation*}
then $U'_{A,q}=A^{\infty}$.
\item If 
\begin{equation}\label{11}
1<q\le R_A:=1+\frac{a_J-a_1}{\max_{j>1} \set{a_j-a_{j-1}}},
\end{equation}
and $p>q$, then $U'_{A,q}\subset U'_{A,p}$.
\end{enumerate}
\end{proposition}

It follows from the proposition that there exist two critical bases $p_A$ and $r_A$ such that $1<p_A\le r_A$, and 
\begin{align*}
q\in (1,p_A)&\Longrightarrow U'_{A,q}\quad\text{is finite;}\\
q\in (p_A,r_A)&\Longrightarrow U'_{A,q}\quad\text{is countably infinite;}\\
q\in (r_A,\infty)&\Longrightarrow U'_{A,q}\quad\text{is uncountable.}
\end{align*}
It was observed without proof by Erd\H os \cite{Erd1996} that either $U'_{A,q}$ is countable, or it has the power of continuum: this property holds without assuming the continuum hypothesis.
A proof was given by Baker \cite{Baker2015}; see also \cite[Theorem 2.3.1, p. 22]{DevKom2016}.

\begin{example}\label{e12}
For $A=\set{0,1}$ we have $p_A=\varphi\approx 1.61803$ (the \emph{Golden Ratio}) and $r_A=q'\approx 1.78723$ (the \emph{Komornik--Loreti constant}, see \cite{KomLor1998}).
The first result was proved by Dar\'oczy et al. \cite{DarJarKat1986}, \cite{DarKat1993} (see also \cite{ErdJooKom1990} and \cite{SidVer1998}), while the second was established by Glendinning and Sidorov \cite{GleSid2001}. 
More precisely, 
\begin{align*}
q\in (1,\varphi]&\Longrightarrow U'_{A,q}=\set{0^{\infty},1^{\infty}};\\
q\in (\varphi,q')&\Longrightarrow U'_{A,q}\quad\text{is countably infinite;}\\
q\in [q',\infty)&\Longrightarrow U'_{A,q}\quad\text{has the power of continuum;}\\
q\in (2,\infty)&\Longleftrightarrow U'_{A,q}=\set{0,1}^{\infty}.
\end{align*}
For example, $\set{0^*(10)^{\infty}, 1^*(01)^{\infty}}\subset U'_{A,q}$ for all $q>\varphi$.

See also de Vries \cite{Dev2009} for stronger set-theoretical results.

Since the critical bases are invariant for non-constant affine transformations of the alphabet, these results remain valid for all two-letter alphabets.
\end{example}

\begin{example}\label{e13}
For more general regular alphabets $A_m=\set{0,1,\ldots, m}$, $m=2,3,\ldots,$ the critical bases  $p_{A_m}$ have been determined by Baker \cite{Baker2014}: they are integer if $m$ is even, and quadratic irrational numbers if $m$ is odd.

The bases $r_{A_m}$ have been determined in \cite[p. 425]{DevKom2009} for $m=2$ and by Kong, Li and Dekking \cite{KongLiDekking2010} for $m\ge 3$. Generalizing the just mentioned theorem of Glendinning and Sidorov, they have proved that the bases $r_{A_m}$ coincide with the generalized Komornik--Loreti constants introduced in   \cite{KomLor2002}, and hence they are transcendental numbers.
For example, we have $r_{A_2}\approx 2.53595.$
\end{example}

As an example we recall some results  concerning the \emph{generalized golden ratios} $p_A$ for all three-letter alphabets.
By an affine transformation it suffices to consider the alphabets $\set{0,1,m}$ with $m\in [2,\infty)$.
Properties (i)-(viii) below have been obtained in \cite{KomLaiPed2010} (see Theorem 1.1, the proof of Lemma 5.3 and Remark 5.12); (ix) is due to Lai \cite{Lai2011}.

\begin{theorem}[KLP]\label{t12}
We consider the alphabets $\set{0,1,m}$ with $m\in [2,\infty)$, and we write $p_m$ instead of $p_{\set{0,1,m}}$ for brevity.

\begin{enumerate}[\upshape (i)]
\item The function $m\mapsto p_m$ is continuous.
\item We have $2\le p_m\le P_m:=1+\sqrt{m/(m-1)}$ for all $m$.
\item $p_m=2\Longleftrightarrow m\in\set{2,4,8,\ldots}$.
\item $C:=\set{m\ :\ p_m=1+\sqrt{m/(m-1)}}$ is a Cantor set. 
\item In each connected component\footnote{The indices $d$ run over a set of sequences closely related to the Sturmian sequences in symbolic dynamics.} $(m_d,M_d)$ of $(2,\infty)\setminus C$ there exists a point $\mu_d$ such that $p_m$ is decreasing in $[m_d,\mu_d]$ and increasing in $[\mu_d,M_d]$.
Moreover, $p_m$ is given by two explicit formulas in $(m_d,\mu_d]$ and $[\mu_d,M_d)$, respectively.
\item If $q\in (1,p_m)$, then $U'_{A,q}=\set{0^{\infty},m^{\infty}}$.
Furthermore, $U'_{A,p_m}=\set{0^{\infty},m^{\infty}}$ if and only if $m\in [m_d,M_d]$ where $(m_d,M_d)$ is a connected component of $(2,\infty)\setminus C$.
\item If $q>p_m$, then $\set{m^*\delta'}\subset U'_{A,q}$ with some sequence $\delta'\ne m^{\infty}$ depending on $m$, to be described later.
\item If $q\in (1,P_m)$, then each element of $U'_{m,q}\setminus\set{0^{\infty}}$ has the form $0^*c$ with some $c\in\set{1,m}^{\infty}$.
\item If $q\in (1,P_m)$, then $U'_{m,q}$ is countable.
\end{enumerate}
\end{theorem}

It follows from these results that 
\begin{equation}\label{12}
2\le p_m\le P_m\le r_m
\end{equation}
for all $m\ge 2$.

An interesting open problem is the determination of the critical bases $r_A$ for all three-letter alphabets. 
The fact that $r_A$ is transcendental for the simplest such alphabet $\set{0,1,2}$ indicates the difficulty of this problem.

A more tractable problem is suggested by property (viii) in Theorem \ref{t12} above: this implies that for $q>p_m$ not only $U'_{m,q}$, but already $U'_{m,q}\cap \set{1,m}^{\infty}$ is infinite.

Motivated by this example we may investigate the size of $U'_{A,q}\cap B^{\infty}$ instead of $U'_{A,q}$, with any given subset $B$ of the alphabet $A$.
We have always $U'_{A,q}\cap B^{\infty}\subset U'_{B,q}$, but the converse inclusion may fail.

To prove the first assertion we observe that if a sequence $(c_i)\in B^{\infty}$ does not belong to $U'_{B,q}$, then there is another sequence $(d_i)\in B^{\infty}$ satisfying $\pi_q(d_i)=\pi_q(c_i)$.
Since $B\subset A$, this shows that $(c_i)$ does not belong to $U'_{A,q}$ either.
The second assertion follows from the following counterexample:

\begin{example}\label{e15}
Let $A=\set{0,1,2}$, $B=\set{0,1}$ and $q=2$.
Then the constant sequence $1^{\infty}$ belongs to $B^{\infty}=\set{0,1}^{\infty}$, but it does not belong to $U'_{A,q}$ because 
\begin{equation*}
\pi_2(1^{\infty})=\pi_2(20^{\infty}).
\end{equation*}
\end{example}

Proposition \ref{p11} also implies the existence of $p_{A,B}, r_{A,B}\in (1,\infty]$ satisfying $p_{A,B}\le r_{A,B}$ and such that 
\begin{align*}
q\in (1,p_{A,B})&\Longrightarrow U'_{A,q}\cap B^{\infty}\quad\text{is finite;}\\
q\in (p_{A,B},r_{A,B})&\Longrightarrow U'_{A,q}\cap B^{\infty}\quad\text{is countably infinite;}\\
q\in (r_{A,B},\infty)&\Longrightarrow U'_{A,q}\cap B^{\infty}\quad\text{is uncountable.}
\end{align*}
If $B$ is empty or has a unique element, then the set $U'_{A,q}\cap B^{\infty}$ has at most one element, so that $p_{A,B}=r_{A,B}=\infty$. 
Otherwise the problem is non-trivial.

For ternary alphabets the above mentioned property may be expressed by the equality
\begin{equation*}
p_{\set{0,1,m},\set{1,m}}=p_{\set{0,1,m}}(=p_m).
\end{equation*}
Motivated by this we focus on the determination of
\begin{equation*}
r_m:=r_{\set{0,1,m},\set{1,m}},
\end{equation*}
i.e., we investigate only unique expansions not containing the zero digit. 
For brevity we write henceforth
\begin{equation*}
V'_{m,q}:=U'_{\set{0,1,m},q}\cap \set{1,m}^{\infty}.
\end{equation*}
We have thus
\begin{align*}
q\in (1,p_m)&\Longrightarrow V'_{m,q}=\set{m^{\infty}}\quad\text{is finite;}\\
q\in (p_m,r_m)&\Longrightarrow V'_{m,q}\quad\text{is countably infinite;}\\
q\in (r_m,\infty)&\Longrightarrow V'_{m,q}\quad\text{is uncountable.}
\end{align*}

Now we are ready to state our results on ternary alphabets.

First we will complete the inequalities \eqref{12}:
\begin{proposition}\label{p13}
We have 
\begin{equation}\label{13}
r_m\le R_m:=1+\frac{m}{m-1}
\end{equation}
for all $m\in [2,\infty)$.
\end{proposition}
Observe that $R_m$ is equal to the right side expression of \eqref{11} for the alphabet $\set{0,1,m}$.

Then we will determine the critical base $r_m$ for $m$ belonging to some special intervals.

We start with the first connected component (see property (v) in  Theorem \ref{t12})
\begin{equation*}
[\mu_d,M_d)=[2,1+\alpha)
\end{equation*}
of $[2,\infty)\setminus C$, where $\alpha=1.32472\ldots$ denotes the first Pisot number.
This corresponds to the smallest ``admissible sequence'' $d=0^{\infty}$, as explained below, in Subsection \ref{ss91} of the Appendix.
We recall (see \eqref{92} below) that $p_m=m$ for all $m\in [2,1+\alpha]$.

\begin{proposition}\label{p14}\mbox{}
If $m\in [2,1+\alpha]=[2,2.32472\ldots]$, then $r_m$ is the unique positive solution $q$ of the equation
\begin{equation*}
\pi_q(m1^{\infty})=m-1.
\end{equation*}
Moreover, $V'_{m,r_m}$ is still countable.
\end{proposition}
See Figure \ref{f11}.

\begin{figure}
\includegraphics[scale=0.8]{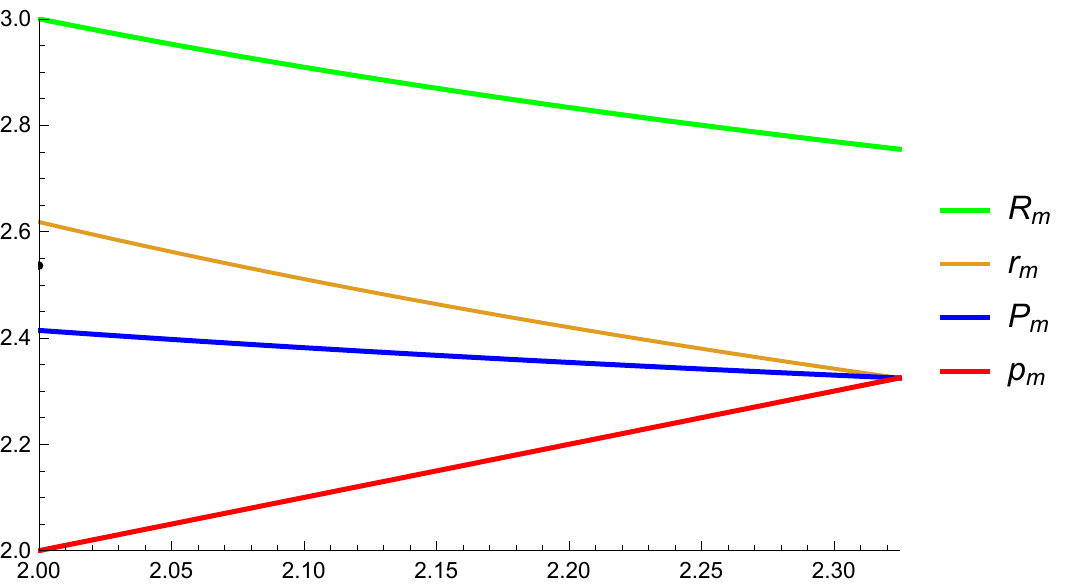}
\caption{Proposition \ref{p14}: $p_m$, $P_m$, $r_m$ and $R_m$ in $[2,2.32472]$}\label{f11}
\end{figure}

\begin{remark}
A direct computation (see \eqref{94} below) shows that 
\begin{equation*}
r_m=\frac{2m-1+\sqrt{4m-3}}{2m-2}
\end{equation*}
in the component $[2,1+\alpha]$.
In particular, $r_2=1+\varphi\approx 2.61803$.
\end{remark}

Next we consider the connected component
\begin{equation*}
(m_d,M_d)=(2.80194\ldots ,4.54646\ldots)
\end{equation*}
of $(2,\infty)\setminus C$; see Figure \ref{f12}.
This corresponds to the smallest ``admissible sequence'' $d=(10)^{\infty}$, as explained in Subsection \ref{ss93} of the Appendix, where the formula of $p_m=\max\set{p_m',p_m''}$ is also given.

\begin{figure}
\includegraphics[scale=0.8]{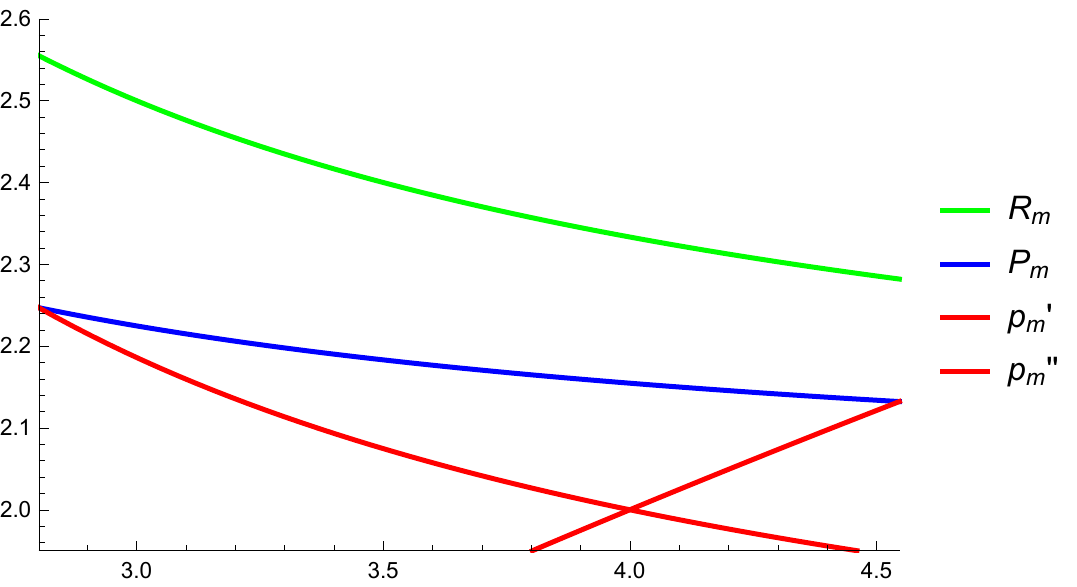}
\caption{Proposition \ref{p15}: $p''_m$, $p'_m$, $P_m$ and $R_m$ in $[2.80194,4.54646]$}\label{f12}
\end{figure}

We will determine $r_m$ for $m$ belonging to three subintervals
\begin{equation*}
[m_d,m_1],\quad [m_2,m_3],\quad [m_4,M_d]
\end{equation*}
of $[m_d,M_d]$,  with 
\begin{equation*}
m_1\approx 2.91286,\quad m_2\approx 2.992,\quad m_3\approx 3.10204,\quad m_4\approx 3.30278.
\end{equation*}
The precise definitions of these numbers will be given during the proof of the following proposition.

\begin{proposition}\label{p15}\mbox{}

\begin{enumerate}[\upshape (i)]
\item If $m\in [m_d,m_1]\approx [2.80194,2.91286]$, then $r_m$ is the unique positive solution of the equation
\begin{equation*}
\pi_q(m^{\infty}-(1m)^{\infty})=1.
\end{equation*}
\item If $m\in [m_2,m_3]\approx [2.992,3.10204]$, then  $r_m$ is the unique positive solution of the equation
\begin{equation*}
\pi_q\left(mm1(m11m)^{\infty}\right)=m-1.
\end{equation*}
\item If $m\in [m_4,M_d]\approx [3.30278,4.54646]$, then $r_m$ is the unique positive solution of the equation
\begin{equation*}
\pi_q(m(m1)^{\infty})=m-1.
\end{equation*}
\end{enumerate}
\end{proposition}

\noindent See Figures \ref{f13}, \ref{f14}, \ref{f15} for the three separate cases and Figure \ref{f16} for a global picture.
Part of Figure \ref{f16} is shown in greater detail in Figure \ref{f17} to understand the situation concerning the curves $r_m$ of the middle and  right intervals.

\begin{figure}
\includegraphics[scale=0.8]{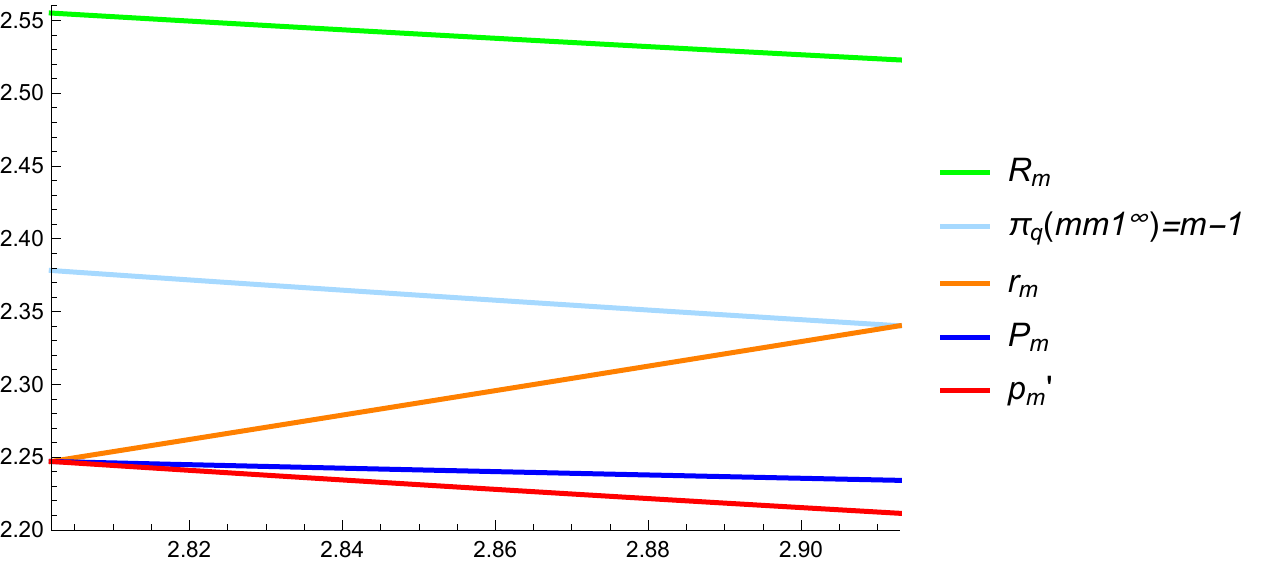}
\caption{Proposition \ref{p15} (i): $p_m$, $P_m$, $r_m$, $R_m$ and an auxiliary curve in $[2.80194,2.91286]$}\label{f13}
\end{figure}

\begin{figure}
\includegraphics[scale=0.8]{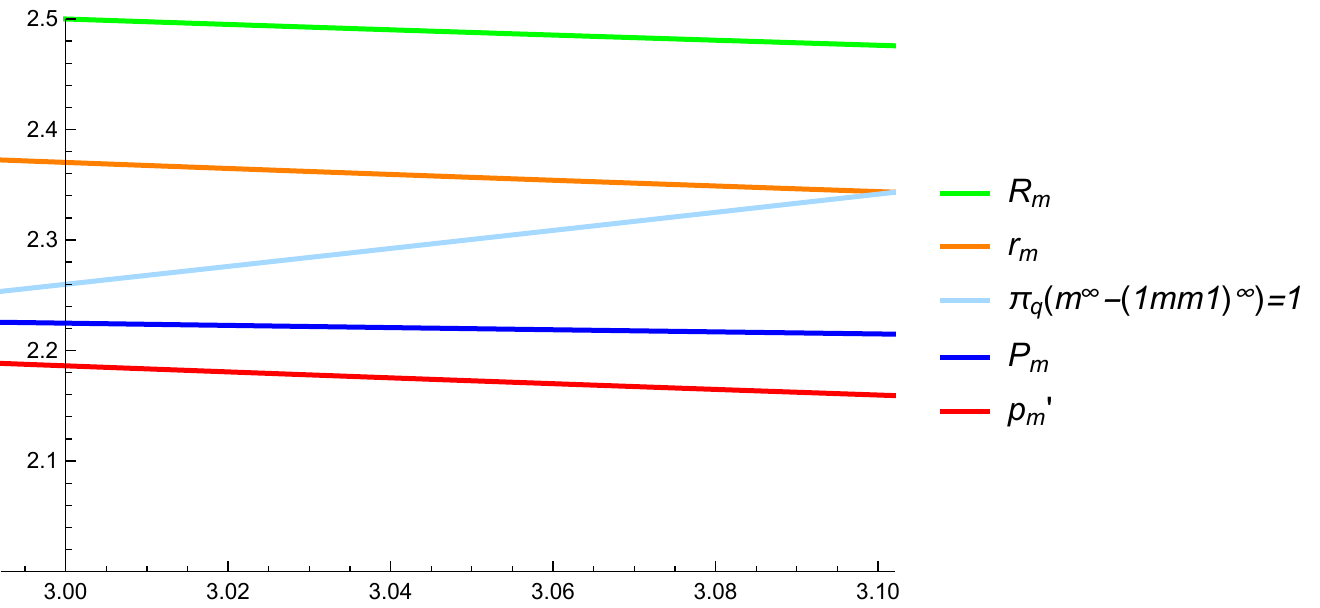}
\caption{Proposition \ref{p15} (ii): $p_m$, $P_m$, $r_m$, $R_m$ and an auxiliary curve in $[2.992,3.10214]$}\label{f14}
\remove{P14-2.pdf=fig2-20160309.pdf}
\remove{Before we had the file f54.pdf labeled as f54.}
\end{figure}

\begin{figure}
\includegraphics[scale=0.8]{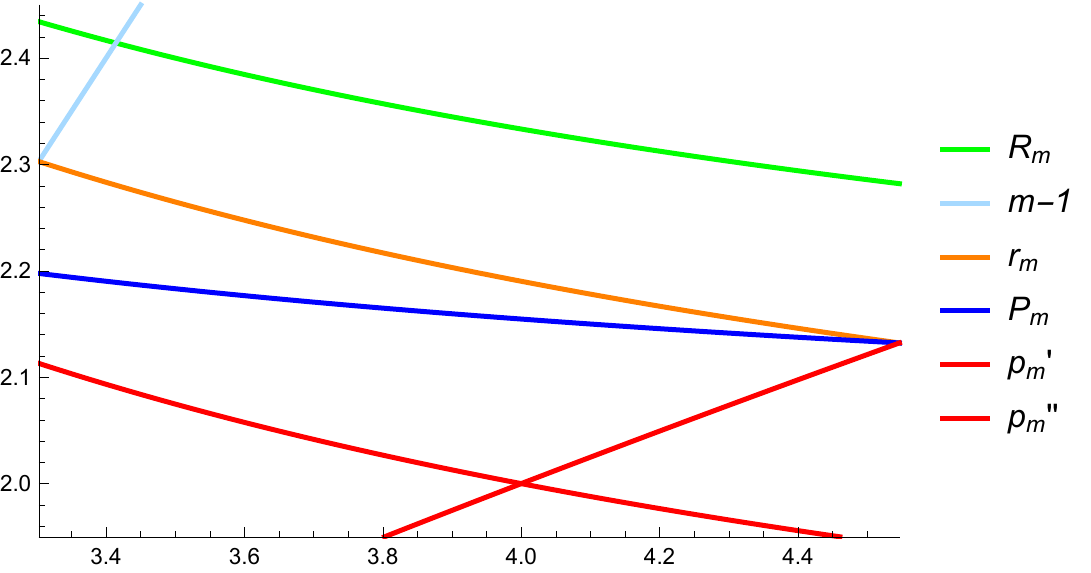}
\caption{Proposition \ref{p15} (iii): $p''_m$, $p'_m$, $P_m$, $r_m$, $R_m$ and an auxiliary curve in $[3.30278,4.54646]$}\label{f15}
\end{figure}

\begin{figure}
\includegraphics[scale=0.8]{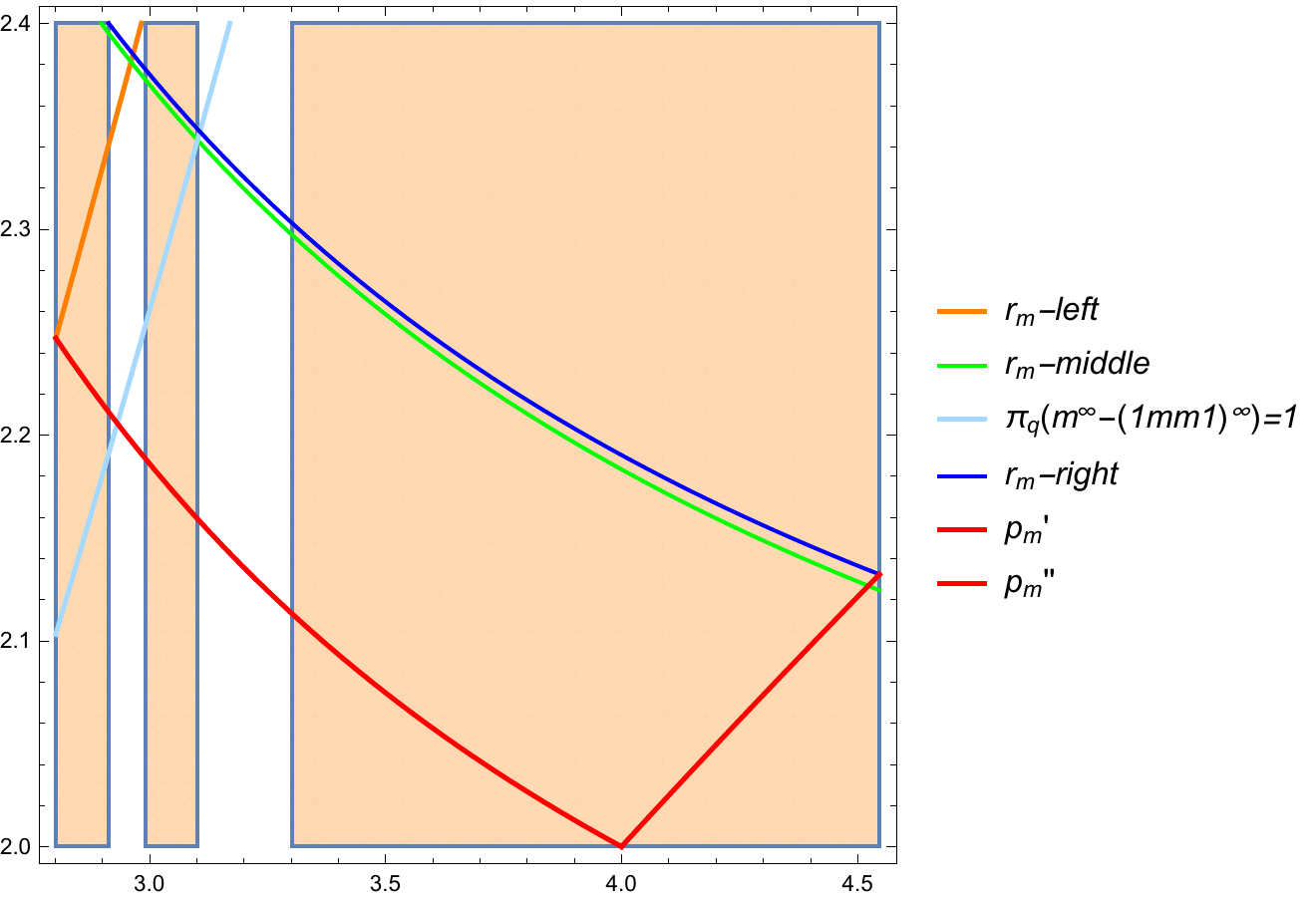}
\caption{Proposition \ref{p15}: global picture in $[3.30278,4.54646]$}\label{f16}
\end{figure}

\begin{figure}
\includegraphics[scale=0.8]{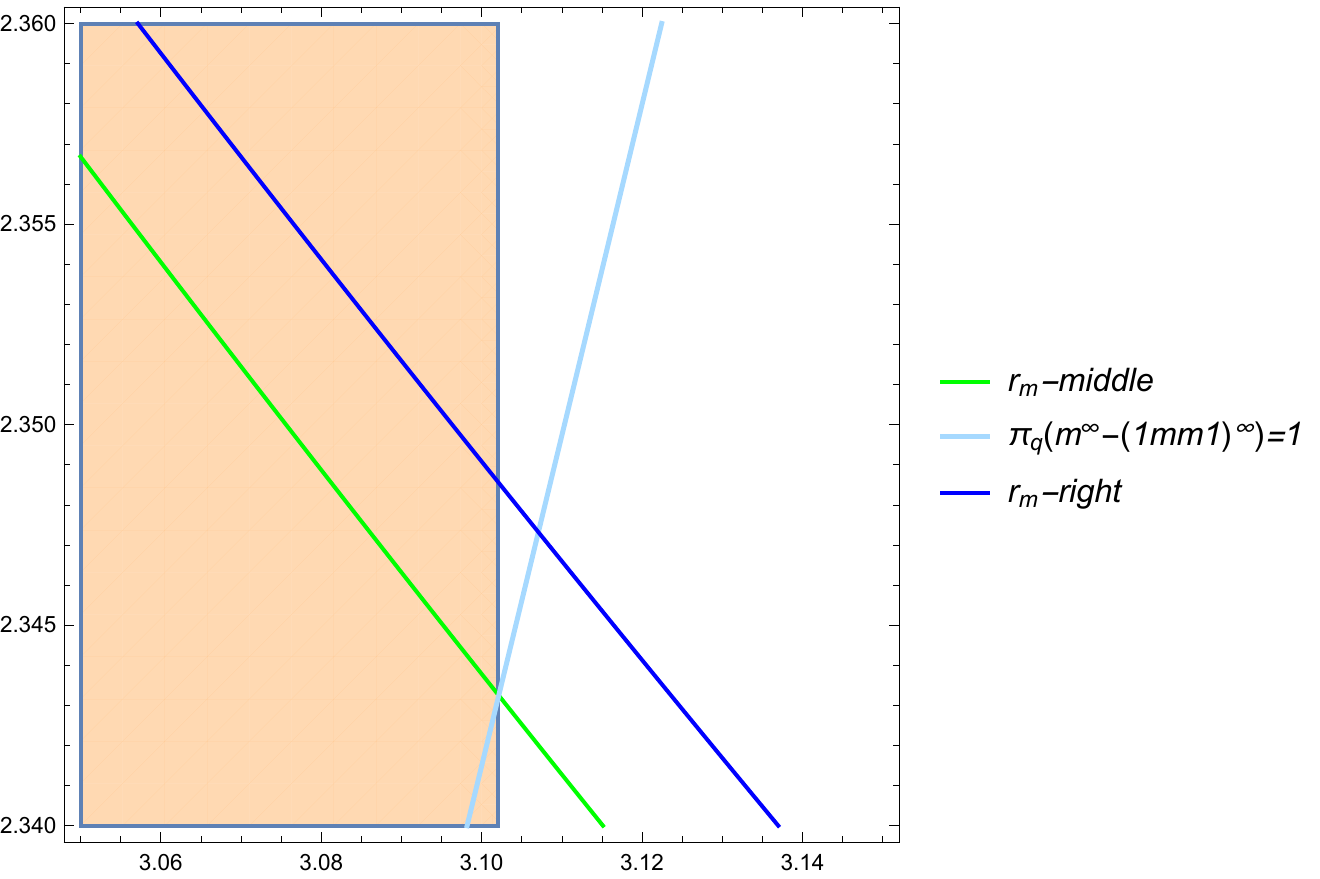}
\caption{Proposition \ref{p15}: the mutual positions of the curves defining $r_m$ in the middle and right zones}\label{f17}
\end{figure}

\begin{remarks}\mbox{}
\begin{enumerate}[\upshape (i)]
\item A straightforward computation shows that $r_m$ is  the unique solution $q>2$ of the polynomial equations
\begin{align*}
&(m-1)q=q^2-1,\\
&(m-1)(q^6-2q^5+q^4-q^3-q^2+2q-1)=q^5+q^3,\\
&(m-1)(q^3-q^2-2q+1)=q^2+q
\end{align*}
in the three cases, respectively.
In particular,
\begin{equation*}
r_m=\frac{m-1+\sqrt{(m-1)^2+4}}{2}
\end{equation*}
in the first case (see \eqref{97}).
\item While $p_m$ is given by two formulas in $[m_d,M_d]$ (see property (v) in Theorem \ref{t12}), $r_m$ is given by three different formulas in different subintervals of $[m_d,M_d]$. 
This indicates that the determination of $r_m$ might be more difficult than that of $p_m$. 
\end{enumerate}
\end{remarks}

In the following sections we prove Propositions \ref{p11}--\ref{p15}.
All proofs are mathematically complete, except that of Propositions \ref{p15} (ii). 
For the latter we admit some intermediate results obtained by symbolic computations and computer simulations; otherwise the proof would become too long.
In order to facilitate the reading, many technical computations are collected in an Appendix at the end of the paper.

\section{Proof of Proposition \ref{p11}}\label{s2}

Let us write $A=\set{a_1<\cdots<a_J}$, $J\ge 2$.
We recall\footnote{See, e.g.,  the proof of \cite[Theorem 2.2]{KomLaiPed2010}.} the following description of $U'_{A,q}$:

\begin{proposition}\label{p21}
Let $(c_i)\in A^{\infty}$ and $q>1$. 
\begin{enumerate}[\upshape (i)]
\item If $(c_i)\in A^{\infty}$ satisfies for some $q>1$ the conditions
\begin{align}
&\sum_{i=1}^\infty \frac{c_{n+i}-a_1}{q^i}  < a_{j+1}-a_j \qtq{whenever}  c_n=a_j<a_J\label{21}
\intertext{and}
&\sum_{i=1}^\infty \frac{a_J-c_{n+i}}{q^i}  < a_j-a_{j-1} \qtq{whenever}  c_n=a_j>a_1,\label{22}
\end{align}
then $(c_i)\in U'_{A,q}$. 
\item Conversely, if 
\begin{equation}\label{23}
1<q\le 1+\frac{a_J-a_1}{\max_{j>1} \set{a_j-a_{j-1}}} (\le J)
\end{equation}
and $(c_i)\in U'_{A,q}$, then the inequalities \eqref{21} and \eqref{22} are satisfied.
\end{enumerate}
\end{proposition}

\begin{proof}[Proof of Proposition \ref{p11}]
\mbox{}
\begin{enumerate}[\upshape (i)]
\item This is proved in \cite[Corollary 2.7]{KomLaiPed2010}.
\item If 
\begin{equation*}
q>1+\frac{a_J-a_1}{\min_{j>1} \set{a_j-a_{j-1}}},
\end{equation*}
then every sequence $(c_i)\in A^{\infty}$ satisfies \eqref{21} and \eqref{22}, so that $U'_{A,q}=A^{\infty}$.
\item We apply Proposition \ref{p21}.
Since $q$ satisfies \eqref{23} by our assumption, $U'_{A,q}$ is characterized by the conditions \eqref{21} and \eqref{22}. 
Since $p>q$, both conditions remain valid by changing $q$ to $p$, and therefore $U'_{A,q}\subset U'_{A,p}$.
\end{enumerate}
\end{proof}

\begin{example}\label{e22}
We illustrate the usefulness of Proposition \ref{p21} by reproving some assertions on the alphabet $A=\set{0,1}$, mentioned above.

It follows from this proposition that in a base $q\in (1,2]$ a sequence $(c_i)\in\set{0,1}^{\infty}$ belongs to $U'_{A,q}$ if and only if
\begin{align}
&\sum_{i=1}^\infty \frac{c_{n+i}}{q^i}  < 1 \qtq{whenever}  c_n=0\label{24}
\intertext{and}
&\sum_{i=1}^\infty \frac{1-c_{n+i}}{q^i}  < 1 \qtq{whenever} c_n=1.\label{25}
\end{align}

If $q\in (1,\varphi]$, then $q^{-1}+q^{-2}\ge 1$, so that both words $011$ and $100$ are forbidden in every sequence $(c_i)\in U'_{A,q}$.
This leaves only the possibilities: $0^{\infty}$, $1^{\infty}$, $0^*(10)^{\infty}$ and $1^*(01)^{\infty}$. 
Since
\begin{equation*}
q^{-1}+q^{-3}+q^{-5}+\cdots\ge 1,
\end{equation*}
none of the sequences $0^*(10)^{\infty}$, $1^*(01)^{\infty}$  satisfies \eqref{24} and \eqref{25} either, so that $U'_{A,q}=\set{0^{\infty},1^{\infty}}$.

If $q\in (\varphi,\infty)$, then all sequences $0^*(10)^{\infty}$, $1^*(01)^{\infty}$ satisfy \eqref{24} and \eqref{25} and hence belong to $U'_{A,q}$.
Therefore $p_A=\varphi$.

On the other hand, if $q>2$, then
\begin{equation*}
q^{-1}+q^{-3}+q^{-5}+\cdots<1,
\end{equation*}
so that \eqref{24} and \eqref{25} are satisfied for all sequences $(c_i)$ of zeros and ones. 
Therefore $U'_{A,q}=A^{\infty}$.
\end{example}

\section{Proof of Proposition \ref{p13}}\label{s3}

We will show that if $q\ge 1+m/(m-1)$ and $k$ is a sufficiently large integer, then 
\begin{equation*}
\set{m^{\ell}1\quad :\quad \ell\ge k}^{\infty}\subset V'_{m,q},
\end{equation*}
and hence $V'_{m,q}$ has the power of continuum.

Applying Proposition \ref{p21} (i) for this special case, it suffices to show that each sequence 
\begin{equation*}
(c_i)\in \set{m^{\ell}1\quad :\quad \ell\ge k}^{\infty}
\end{equation*}
satifies the following four conditions:
\begin{align*}
&\sum_{i=1}^\infty \frac{c_{n+i}}{q^i}  < 1 \qtq{whenever}  c_n=0;\\
&\sum_{i=1}^\infty \frac{c_{n+i}}{q^i}  < m-1 \qtq{whenever}  c_n=1;\\
&\sum_{i=1}^\infty \frac{m-c_{n+i}}{q^i}  < 1 \qtq{whenever} c_n=1;\\
&\sum_{i=1}^\infty \frac{m-c_{n+i}}{q^i}  < m-1 \qtq{whenever}  c_n=m.
\end{align*}

The first condition is trivially satisfied  because there is no zero digit.
The second condition follows by using our assumption on $q$:
\begin{equation*}
\sum_{i=1}^\infty \frac{c_{n+i}}{q^i}<\sum_{i=1}^\infty \frac{m}{q^i}=\frac{m}{q-1}\le m-1.
\end{equation*}
For the proof of the fourth condition first we observe that
\begin{equation*}
\sum_{i=1}^\infty \frac{m-c_{n+i}}{q^i}\le \pi_q\left(((m-1)0^k)^\infty\right)=(m-1)\frac{q^k}{q^{k+1}-1}
\end{equation*}
if $c_n=m$.
We conclude by observing that our assumption $q\ge 1+m/(m-1)$ implies that $q>2$, and then $q^k/(q^{k+1}-1)<1$ for all $k\ge 0$.

Finally, we have
\begin{equation*}
\sum_{i=1}^\infty \frac{m-c_{n+i}}{q^i}\le \pi_q\left((0^k(m-1))^\infty\right)=\frac{m-1}{q^{k+1}-1}
\end{equation*}
if $c_n=1$.
Since $q>1$, the third condition follows by choosing a sufficiently large integer $k$ satisfying $q^{k+1}>m$.

\section{Preliminary lemmas on ternary alphabets}\label{s4}

We establish three elementary lemmas on ternary alphabets that will be frequently used in the sequel.

\begin{lemma}\label{l41}
Let $q>2$ and $(a_i), (b_i)\in\set{1,m}^{\infty}$. 
If $(a_i)<(b_i)$, then $\pi_q(a_i)<\pi_q(b_i)$.
\end{lemma}

\begin{proof}
If $n$ is the first index such that $a_n<b_n$, then 
\begin{equation*}
\pi_q(b_i)-\pi_q(a_i) 
\ge \frac{m-1}{q^n}-\sum_{i=n+1}^{\infty}\frac{m-1}{q^i} 
=\frac{m-1}{q^n}\left(1-\frac{1}{q-1} \right) >0.\qedhere
\end{equation*}
\end{proof}

\begin{lemma}\label{l42}\mbox{}
\begin{enumerate}[\upshape (i)]
\item If $(c_i)\in\set{1,m}^{\infty}$ satisfies for some $q>2$ the conditions
\begin{equation}\label{41}
\pi_q(c_{n+i})<m-1 
\qtq{and}
\pi_q(m-c_{n+i})<1
\end{equation} 
whenever $c_n=1$, then $(c_i)\in V'_{m,q}$.
\item Conversely, if $q\in(2,R_m]$ and $(c_i)\in V'_{m,q}$, then  the inequalities \eqref{41} are satisfied.
\end{enumerate}
\end{lemma}

\begin{proof}\mbox{}
\begin{enumerate}[\upshape (i)]
\item Proposition \ref{p21} (i) contains two other conditions: $\pi_q(c_{n+i})<1$ whenever $c_n=0$, and $\pi_q(m-c_{n+i})<m-1$ whenever $c_n=m$. 
The first condition may be omitted because there are no zero digits. 
Since $q>2$, the second condition is automatically satisfied, even if $c_n\ne m$:
\begin{equation*}
\pi_q(m-c_{n+i})\le \pi_q((m-1)^{\infty})=\frac{m-1}{q-1}<m-1
\end{equation*}
for all $n$.
\item This is a special case of Proposition \ref{p21} (ii).
\end{enumerate}
\end{proof}

Henceforth a finite block $a_0\cdots a_k$ of digits is called \emph{forbidden} if it cannot occur in any sequence $(c_i)\in V'_{m,q}$.

\begin{lemma}\label{l43}
Let $q\in(2,R_m]$ and  $a_1\cdots a_k\in\set{1,m}^k$.
If 
\begin{equation*}
\pi_q(a_1\cdots a_k1^{\infty})\ge m-1 
\qtq{or of}
\pi_q(a_1\cdots a_km^{\infty})\le \frac{m}{q-1}-1,
\end{equation*}
then the block $1a_1\cdots a_k$ is forbidden.
\end{lemma}

\begin{proof}
If $(c_i)\in V'_{m,q}$ and $c_mc_{m+1}\cdots c_{m+k}=1a_1\cdots a_k$ for some $m$, then applying  the preceding two lemmas we obtain
\begin{align*}
\pi_q(a_1\cdots a_k1^{\infty})
&=\pi_q(c_{m+1}\cdots c_{m+k}1^{\infty})\\ 
&\le \pi_q(c_{m+1}\cdots c_{m+k}c_{m+k+1}\cdots)<m-1
\end{align*}
and 
\begin{align*}
\pi_q(a_1\cdots a_km^{\infty})
&=\pi_q(c_{m+1}\cdots c_{m+k}m^{\infty})\\
&\ge \pi_q(c_{m+1}\cdots c_{m+k}c_{m+k+1}\cdots)>\frac{m}{q-1}-1.\qedhere
\end{align*}
\end{proof}

\section{Proof of Proposition \ref{p14}}\label{s5}

We recall that in this section we have $d=0^{\infty}$, $\mu_d=2$ and $M_d\approx  2.32472$.

Henceforth we will write $f\sim g$ if $f$ and $g$ have the same sign.

We need a lemma:

\begin{lemma}\label{l51}
For each $m\in [2,M_d]$ there exists a number $r(m)>1$ such that 
\begin{equation}\label{51}
\pi_q(m1^{\infty}) -(m-1)\sim r(m)-q.
\end{equation}

Furthermore, 
\begin{equation*}
P_m\le r(m)<R_m,
\end{equation*}
and the equality $r(m)=P_m$ holds only if $m=M_d$.
\end{lemma}

\begin{proof}
The first assertion follows by observing that $\pi_q(m1^{\infty})$ is a continuous decreasing function of $q\in(0,\infty)$, and 
\begin{equation*}
\lim_{q\to 1}\pi_q(m1^{\infty})=\infty,\quad \lim_{q\to \infty}\pi_q(m1^{\infty})=0.
\end{equation*}

The second assertion follows from  \eqref{94} and the formulas \eqref{95}, \eqref{96} of the Appendix.  
\end{proof}

\begin{proof}[Proof of Proposition \ref{p14}]\mbox{}

\emph{First step.} If $1<q\le r(m)$, then 
\begin{equation*}
\pi_q(m1^{\infty})\ge m-1.
\end{equation*}
Since $2<P_m\le r(m)<R_m$ by \eqref{51}, by Proposition \ref{p11} (iii) it is sufficient to consider the case $2<q\le r(m)$.

Then we may apply Lemma \ref{l43} to infer that $1m$ is a forbidden block. 
Hence $V'_{m,q}\subset\set{m^{\infty},m^*1^{\infty}}$, and thus  $V'_{m,q}$ is countable.
\medskip 

\emph{Second step.} 
If $q>r(m)$, then $V'_{m,q}$ has the power of continuum. 

Indeed, our assumption on $q$ implies that $\pi_q(m1^{\infty})< m-1$.
There exists therefore a positive integer $k$ satisfying 
\begin{equation}\label{52}
\pi_q((m1^k)^{\infty})<m-1.
\end{equation} 
We complete the proof by showing that each sequence
\begin{equation*}
(c_i)\in\set{m1^{\ell}\ :\ \ell=k,k+1,\ldots}^{\infty}
\end{equation*}
belongs to $V'_{m,q}$.

Since $q>P_m\ge 2$, it is sufficient to check the  conditions of Lemma \ref{l42}.
In fact, the conditions \eqref{41} are satisfied for all $n$: we have 
\begin{equation*}
\pi_q(c_{n+i})\le \pi_q((m1^k)^{\infty})<m-1
\end{equation*}
by the lexicographic inequality $(c_{n+i})\le (m1^k)^{\infty}$, Lemma \ref{l41} and \eqref{52},  and 
\begin{equation*}
\pi_q(m-c_{n+i})\le \pi_q((m-1)^{\infty})=\frac{m-1}{q-1}<1
\end{equation*}
because $c_{n+i}\ge 1$ for all $n$, and $q>m$.
\medskip 

It follows from the preceding steps that $r_m=r(m)$.
\end{proof}

\section{Proof of Proposition \ref{p15} (i)}\label{s6}

We recall that in this section $d=(10)^{\infty}$ and $m_d\approx 2.80194$.
First we define $m_1$:

\begin{lemma}\label{l61}\mbox{}

\begin{enumerate}[\upshape (i)]
\item The system of equations 
\begin{equation*}
\pi_q(m^{\infty}-(1m)^{\infty})=1\qtq{and} \pi_q(mm1^{\infty})=m-1
\end{equation*}
has a unique solution $(m_1,q_1)$ with $q_1>2$.
We have  
\begin{equation*}
m_1\approx 2.91286\qtq{and} q_1\approx 2.34018.
\end{equation*}

\item For each $m\in [m_d,m_1]$ there exists a number $r(m)>1$ such that 
\begin{align*}
1<q<r(m)&\Longrightarrow \pi_q(m^{\infty}-(1m)^{\infty})>1\qtq{and} \pi_q(mm1^{\infty})>m-1,\\
q>r(m)&\Longrightarrow \pi_q(m^{\infty}-(1m)^{\infty})<1.
\end{align*}

\item We have
\begin{equation}\label{61}
P_m\le r(m)<R_m,
\end{equation}
and the equality $r(m)=P_m$ holds only if $m=m_d$.
\end{enumerate}
\end{lemma}
\noindent See Figure \ref{f13}.

\begin{proof}\mbox{}
\begin{enumerate}[\upshape (i)]
\item Since
\begin{equation*}
\pi_q(m^{\infty}-(1m)^{\infty})=\pi_q\left(((m-1)0)^{\infty} \right) =\frac{(m-1)q}{q^2-1}
\end{equation*}
and
\begin{equation*}
\pi_q(mm1^{\infty})=\frac{m-1}{q}+\frac{m-1}{q^2}+\frac{1}{q-1},
\end{equation*}
the system of equations is equivalent to 
\begin{equation*}
m-1=\frac{q^2-1}{q}\qtq{and} m-1=\frac{q^2}{(q-1)(q^2-q-1)}.
\end{equation*}
Eliminating $m$ we obtain the equation 
\begin{equation*}
q^2(q-1)(q^2-q-3)=1.
\end{equation*}
The left hand side is increasing in $[2,\infty)$ and changes sign in $(2,3)$, hence it has a unique solution $q_1>2$, satisfying $q_1\in (2,3)$.

Then
\begin{equation*}
m_1=1+\frac{q_1^2-1}{q_1}=1+q_1-\frac{1}{q_1}\in \left(1+2-\frac{1}{2},1+3-\frac{1}{3}\right)
\end{equation*}
because the function $q\mapsto 1+q-1/q$ is increasing for $q>0$.

The numerical values are obtained by evaluating the root $q_1$ of the above polynomial equation.

\item The function
\begin{equation*}
q\mapsto \pi_q(m^{\infty}-(1m)^{\infty})=(m-1)\pi_q\left((10)^{\infty} \right)
\end{equation*}
is continuous and decreasing in $(1,\infty)$, and 
\begin{equation*}
\lim_{q\to 1}\pi_q(m^{\infty}-(1m)^{\infty})=\infty,\quad \lim_{q\to \infty}\pi_q(m^{\infty}-(1m)^{\infty})=0.
\end{equation*}
Therefore there exists a unique number $r(m)\in (1,\infty)$ satisfying
\begin{equation*}
\pi_q(m^{\infty}-(1m)^{\infty})-1\sim r(m)-q.
\end{equation*}
Since the function $q\mapsto \pi_q(mm1^{\infty})$ is also decreasing in $q$, it remains to show that 
\begin{equation*}
\pi_{r(m)}(mm1^{\infty})\ge m-1.
\end{equation*}

Denoting by $q=f(m)$ and $q=g(m)$ the solutions of the equations 
\begin{equation*}
\pi_q(m^{\infty}-(1m)^{\infty})=1\qtq{and} \pi_q(mm1^{\infty})=m-1,
\end{equation*}
respectively, it suffices to show by monotonicity that $f\le g$ in $[m_d,m_1]$.
Since $f(m_1)=g(m_1)$ by (i), it suffices to prove that $f$ is increasing and $g$ is decreasing  in $[m_d,m_1]$.

The increasingness of $f$ follows from the explicit formula 
\begin{equation*}
f(m)=\frac{m-1+\sqrt{(m-1)^2+4}}{2}
\end{equation*}
(see \eqref{912}).

Furthermore, since 
\begin{equation*}
\pi_q(mm1^{\infty})=m-1\Longleftrightarrow \frac{1}{m-1}=q-2+q^{-2}
\end{equation*}
(see \eqref{915}) and since the right hand side has a positive derivative
\begin{equation*}
(q-2+q^{-2})'=1-2q^{-3}>0
\end{equation*}
for all $q>2$, $m$ is a decreasing function of $g(m)$, and hence its inverse function $g$ is also decreasing.

\item This follows from \eqref{912}, \eqref{913} and \eqref{914} in the Appendix.
\end{enumerate}
\end{proof}

\begin{proof}[Proof of Proposition \ref{p15} (i)]\mbox{}

\emph{First step.} If $1<q\le r(m)$, then $V'_{m,q}$ is countable. 

Indeed, it follows from our assumptions that
\begin{equation*}
\pi_q(mm1^{\infty})\ge m-1
\qtq{and}
\pi_q\left(m^{\infty}-(1m)^{\infty}\right)\ge 1.
\end{equation*}
We claim that the blocks $1mm$ and $11$ are forbidden.

As in the proof of Proposition \ref{p14}, we may assume that $2<q\le r(m)$.
If $(c_i)\in\set{1,m}^{\infty}$ contains a block $c_nc_{n+1}c_{n+2}=1mm$, then 
\begin{equation*}
\pi_q(c_{n+i})\ge \pi_q(mm1^{\infty})\ge m-1
\end{equation*}
by our assumption, and hence $(c_i)\notin V'_{m,q}$ by \eqref{61} and Lemma \ref{l42}.

Next assume that $(c_i)\in\set{1,m}^{\infty}$ does not contain any block $1mm$, but it contains a block $c_nc_{n+1}=11$.
Then using also the preceding observation, 
\begin{equation*}
\pi_q(m-c_{n+i})\ge \pi_q\left( m^{\infty}-(1m)^{\infty}\right)\ge 1
\end{equation*}
by hypothesis, and therefore $(c_i)\notin V'_{m,q}$ again.

It follows that $V'_{m,q}\subset\set{m^{\infty}, m^*(1m)^{\infty}}$, and hence it is countable. 
\medskip 

\emph{Second step.} If $q>r(m)$, then $V'_{m,q}$ has the power of continuum. 

Indeed, it follows from our assumption that $\pi_q(m^{\infty}-(1m)^{\infty})<1$. 
There exists therefore a (sufficiently large) positive integer $k$ satisfying
\begin{equation*}
\pi_q\left(m^{\infty}-\left((1m)^k1 \right)^{\infty}\right)<1.
\end{equation*} 
We complete the proof by showing the inclusion
\begin{equation*}
\set{(1m)^{\ell}1\ :\ \ell\ge k}^{\infty}\subset V'_{m,q}.
\end{equation*}

We have to check that each sequence $(c_i)$ in the left hand side set satisfies the conditions of Lemma \ref{l42}.
For this we use the sequence $\delta'=(m1)^{\infty}$, introduced at the beginning of Section 4 in \cite{KomLaiPed2010}; see also \cite[Lemma 3.9]{KomLaiPed2010}.

If $c_n=1$, then \cite[Lemma 5.11]{KomLaiPed2010} and the relation $q>p_m$ imply $\delta'\in V'_{m,q}$ and hence
\begin{equation*}
\pi_q(c_{n+i})<\pi_q\left((m1)^{\infty}\right)=\pi_q(\delta')<m-1,
\end{equation*}
while
\begin{equation*}
\pi_q(m-c_{n+i})\le \pi_q\left(m-\left((1m)^k1 \right)^{\infty}\right)<1
\end{equation*}
by the choice of $k$.
\medskip 

It follows from the preceding steps that $r_m=r(m)$. 
\end{proof}

\section{Proof of Proposition \ref{p15} (ii)}\label{s7}

In this section we still have $d=(10)^{\infty}$, $m_d\approx 2.80194$ and $M_d\approx 4.54646]$.
We consider the subinterval $[m_2,m_3]:=[2.992,3.10214]$ of $(m_d,M_d)$.

\begin{lemma}\label{l71}
For each $m\in [m_2,m_3]$ there exists a number $r(m)>1$ such that
\begin{align*}
1<q\le r(m)&\Longrightarrow \pi_q\left(mm1(m11m)^{\infty}\right)\ge m-1,\\
q>r(m)&\Longrightarrow \pi_q\left(mm1(m11m)^{\infty}\right)<m-1\quad\text{and}\\
&\hspace*{19mm}\pi_q(m^{\infty}-(1mm1)^{\infty})<1.
\end{align*}
Furthermore, $P_m<r(m)<R_m$.
\end{lemma}
\noindent See Figure \ref{f14}.

\begin{proof}
The existence of $r(m)>1$ satisfying
\begin{align*}
&\Longrightarrow \pi_q\left(mm1(m11m)^{\infty}\right)>m-1\qtq{if}1<q<r(m),\\
&\Longrightarrow \pi_q\left(mm1(m11m)^{\infty}\right)<m-1\qtq{if}q>r(m)
\end{align*}
follows by monotonicity as in the proof of Lemma \ref{l51}.
The inequality 
\begin{equation*}
\pi_q(m^{\infty}-(1mm1)^{\infty})<1\qtq{if}q>r(m)
\end{equation*}
follows from the choice of $m_3$, see Figure \ref{f14}.
The same figure also shows the inequalities $P_m<r(m)<R_m$.
\end{proof}

\begin{proof}[Proof of Proposition \ref{p15} (ii)]
Fix $m\in [m_2,m_3]$ arbitrarily.
A numerical computation suggests that for each fixed $m\in [m_2,m_3]$, if $q$ is sufficiently small, then all elements of $V'_{m,q}$ belong to the set
\begin{equation*}
\set{1m(11mm)^k\ :\ k=1,2,\ldots}^{\infty}.
\end{equation*}
In this set the words $111$, $1mmm$, $11m11$ do not occur, and we may find similarly forbidden words of length $5,6,\ldots .$
As we show below, a careful study of the forbidden words of length $\le 7$ allows us to determine the critical base $r_m$.
\medskip 

\emph{First step.}
If $2<q\le r(m)$, then the application of Lemma \ref{l43} shows that the following seven words are forbidden:
\remove{For $m=3$ and $q=2.37034$} 
\begin{align*}
&111,\\ 
&1mmm,\\
&11m11,\\ 
&11m1m1,\\ 
&1mm1mm,\\ 
&11m1mm1,\\ 
&1mm1m1m.
\end{align*}
Using the theory of finite automata (see Figure \ref{f71}) we infer from this that each sequence of $V'_{q,m}$ appears in the following list:
\begin{align*}
&m^{\infty},\\ 
&m^*(1m)^{\infty},\\ 
&m^*(1m)^*11mm1\set{1mm1,m11mm1}^{\infty},\\
&m^*(1m)^*1mm1\set{1mm1,m11mm1}^{\infty}.
\end{align*}

\begin{figure}
\definecolor{white}{RGB}{255,255,255}
\centering
\scalebox{0.9}{\begin{tikzpicture}[y=0.80pt, x=0.8pt,yscale=-1, inner sep=0pt, outer sep=0pt,
draw=black,fill=black,line join=miter,line cap=rect,miter limit=10.00,line width=0.800pt]
  \begin{scope}[shift={(56.0,-188.0)},draw=white,fill=white]
    \path[fill,rounded corners=0.0000cm] (-56.0000,188.0000) rectangle
      (403.0000,398.0000);
  \end{scope}
  \begin{scope}[cm={{1.0,0.0,0.0,1.0,(56.0,-188.0)}},line cap=butt,miter limit=1.45,line width=1.600pt]
    \path[draw] (15.0000,367.7871) circle (0.4233cm);
    \path[draw] (116.9032,367.7871) circle (0.4233cm);
    \path[draw] (233.9916,367.7871) circle (0.4233cm);
    \path[draw] (116.9032,224.7135) circle (0.4233cm);
    \path[draw] (233.9916,224.7135) circle (0.4233cm);
    \path[draw] (372.3973,224.7135) circle (0.4233cm);
    \path[draw] (233.9916,288.2503) circle (0.4233cm);
    \path[draw] (372.3973,367.7871) circle (0.4233cm);
    \path[draw] (303.1945,367.7871) circle (0.4233cm);
    \path[draw] (22.5000,354.7851) -- (22.5000,346.2861) -- (23.1250,342.5677) --
      (25.0000,339.9118) -- (28.1250,338.3183) -- (32.5000,337.7871) --
      (35.0000,337.7871) -- (39.3750,338.4121) -- (42.5000,340.2871) --
      (44.3750,343.4121) -- (45.0000,347.7871) -- (45.0000,350.2871) --
      (44.0003,357.7871) -- (42.7506,359.6621) -- (41.0010,360.2871) --
      (37.0020,360.2871);
    \path[fill] (28.0020,360.2871) -- (41.5020,365.9121) -- (38.1270,360.2871) --
      (41.5020,354.6621) -- cycle;
    \path[fill,line width=0.800pt] (28.1484,330.2558) node[above right] (text42)
      {$m$};
    \path[draw] (30.0000,367.7871) -- (92.9033,367.7871);
    \path[fill] (101.9033,367.7871) -- (88.4033,362.1621) -- (91.7783,367.7871) --
      (88.4033,373.4121) -- cycle;
    \path[fill,line width=0.800pt] (62.1577,360.2558) node[above right] (text48)
      {$1$};
    \path[draw] (131.9032,367.7871) -- (209.9916,367.7871);
    \path[fill] (218.9916,367.7871) -- (205.4916,362.1621) -- (208.8666,367.7871) --
      (205.4916,373.4121) -- cycle;
    \path[fill,line width=0.800pt] (171.6535,360.2558) node[above right] (text54)
      {$1$};
    \path[draw] (110.2712,354.3441) -- (90.6507,304.3098) -- (88.6049,294.9252) --
      (89.7687,285.3909) -- (92.7313,275.1091) -- (99.5192,256.3433) --
      (104.4660,245.0733);
    \path[fill] (108.0833,236.8322) -- (97.5067,246.9330) -- (104.0138,246.1034) --
      (107.8080,251.4546) -- cycle;
    \path[fill,line width=0.800pt] (76.7827,323.8662) node[above right] (text60)
      {$m$};
    \path[draw] (122.8588,238.4951) -- (136.4920,259.1536) -- (140.6230,267.9134) --
      (142.0000,277.5000) -- (142.0000,296.0000) -- (141.0791,305.8242) --
      (138.3164,315.2968) -- (135.1836,323.2032) -- (129.0570,339.2021) --
      (126.6139,345.9042);
    \path[fill] (123.5317,354.3600) -- (133.4399,343.6028) -- (126.9992,344.8473) --
      (122.8702,339.7499) -- cycle;
    \path[fill,line width=0.800pt] (147,291.2852) node[above right] (text66) {$1$};
    \path[draw] (131.9032,224.7135) -- (209.9916,224.7135);
    \path[fill] (218.9916,224.7135) -- (205.4916,219.0885) -- (208.8666,224.7135) --
      (205.4916,230.3385) -- cycle;
    \path[fill,line width=0.800pt] (169.8459,217.1823) node[above right] (text72)
      {$m$};
    \path[draw] (248.9916,224.7135) -- (348.3973,224.7135);
    \path[fill] (357.3973,224.7135) -- (343.8973,219.0885) -- (347.2723,224.7135) --
      (343.8973,230.3385) -- cycle;
    \path[fill,line width=0.800pt] (299.4005,217.1823) node[above right] (text78)
      {$1$};
    \path[draw] (361.9680,235.4945) -- (250.6785,350.5374);
    \path[fill] (244.4209,357.0061) -- (257.8501,351.2141) -- (251.4607,349.7289) --
      (249.7644,343.3922) -- cycle;
    \path[fill,line width=0.800pt] (299.4005,282.7297) node[above right] (text84)
      {$1$};
    \path[draw] (233.9916,352.7871) -- (233.9916,312.2503);
    \path[fill] (233.9916,303.2503) -- (228.3666,316.7503) -- (233.9916,313.3753) --
      (239.6166,316.7503) -- cycle;
    \path[fill,line width=0.800pt] (217.7885,332.5538) node[above right] (text90)
      {$m$};
    \path[draw] (233.9916,273.2503) -- (233.9916,248.7135);
    \path[fill] (233.9916,239.7135) -- (228.3666,253.2135) -- (233.9916,249.8385) --
      (239.6166,253.2135) -- cycle;
    \path[fill,line width=0.800pt] (217.7885,261.0171) node[above right] (text96)
      {$m$};
    \path[draw] (372.3973,239.7135) -- (372.3973,343.7871);
    \path[fill] (372.3973,352.7871) -- (378.0223,339.2871) -- (372.3973,342.6621) --
      (366.7723,339.2871) -- cycle;
    \path[fill,line width=0.800pt] (356.1942,300.7855) node[above right] (text102)
      {$m$};
    \path[draw] (357.3973,367.7871) -- (327.1945,367.7871);
    \path[fill] (318.1945,367.7871) -- (331.6945,373.4121) -- (328.3195,367.7871) --
      (331.6945,362.1621) -- cycle;
    \path[fill,line width=0.800pt] (334.002,360.2558) node[above right] (text108)
      {$1$};
    \path[draw] (288.1945,367.7871) -- (257.9916,367.7871);
    \path[fill] (248.9916,367.7871) -- (262.4916,373.4121) -- (259.1166,367.7871) --
      (262.4916,362.1621) -- cycle;
    \path[fill,line width=0.800pt] (264.7991,360.2558) node[above right] (text114)
      {$1$};
    \path[draw] (-14.1508,342.5950) -- (-3.1587,352.0944);
    \path[fill] (3.6508,357.9791) -- (-2.8855,344.8960) -- (-4.0099,351.3588) --
      (-10.2414,353.4079) -- cycle;
  \end{scope}
\end{tikzpicture}}
\caption{Automaton for the seven forbidden words}\label{f71}
\end{figure}

\begin{figure}
\centering
\scalebox{0.9}{
\definecolor{white}{RGB}{255,255,255}
\begin{tikzpicture}[y=0.80pt, x=0.8pt,yscale=-1, inner sep=0pt, outer sep=0pt,
draw=black,fill=black,line join=miter,line cap=rect,miter limit=10.00,line width=0.800pt]
  \begin{scope}[shift={(56.0,-189.0)},draw=white,fill=white]
    \path[fill,rounded corners=0.0000cm] (-56.0000,189.0000) rectangle
      (403.0000,398.0000);
  \end{scope}
  \begin{scope}[cm={{1.0,0.0,0.0,1.0,(56.0,-189.0)}},line cap=butt,miter limit=1.45,line width=1.600pt]
    \path[draw] (15.0000,367.7871) circle (0.4233cm);
    \path[draw] (116.9032,367.7871) circle (0.4233cm);
    \path[draw] (233.9916,367.7871) circle (0.4233cm);
    \path[draw] (116.9032,224.7135) circle (0.4233cm);
    \path[draw] (233.9916,224.7135) circle (0.4233cm);
    \path[draw] (372.3973,224.7135) circle (0.4233cm);
    \path[draw] (233.9916,288.2503) circle (0.4233cm);
    \path[draw] (22.5000,354.7851) -- (22.5000,346.2861) -- (23.1250,342.5677) --
      (25.0000,339.9118) -- (28.1250,338.3183) -- (32.5000,337.7871) --
      (35.0000,337.7871) -- (39.3750,338.4121) -- (42.5000,340.2871) --
      (44.3750,343.4121) -- (45.0000,347.7871) -- (45.0000,350.2871) --
      (44.0003,357.7871) -- (42.7506,359.6621) -- (41.0010,360.2871) --
      (37.0020,360.2871);
    \path[fill] (28.0020,360.2871) -- (41.5020,365.9121) -- (38.1270,360.2871) --
      (41.5020,354.6621) -- cycle;
    \path[fill,line width=0.800pt] (28.1484,331.2558) node[above right] (text38)
      {$m$};
    \path[draw] (30.0000,367.7871) -- (92.9033,367.7871);
    \path[fill] (101.9033,367.7871) -- (88.4033,362.1621) -- (91.7783,367.7871) --
      (88.4033,373.4121) -- cycle;
    \path[fill,line width=0.800pt] (62.1577,361.2558) node[above right] (text44)
      {$1$};
    \path[draw] (131.9032,367.7871) -- (209.9916,367.7871);
    \path[fill] (218.9916,367.7871) -- (205.4916,362.1621) -- (208.8666,367.7871) --
      (205.4916,373.4121) -- cycle;
    \path[fill,line width=0.800pt] (171.6535,361.2558) node[above right] (text50)
      {$1$};
    \path[draw] (110.2712,354.3441) -- (90.6507,304.3098) -- (88.6049,294.9252) --
      (89.7687,285.3909) -- (92.7313,275.1091) -- (99.5192,256.3433) --
      (104.4660,245.0733);
    \path[fill] (108.0833,236.8322) -- (97.5067,246.9330) -- (104.0138,246.1034) --
      (107.8080,251.4546) -- cycle;
    \path[fill,line width=0.800pt] (77.7827,323.8662) node[above right] (text56)
      {$m$};
    \path[draw] (122.8588,238.4951) -- (136.4920,259.1536) -- (140.6230,267.9134) --
      (142.0000,277.5000) -- (142.0000,296.0000) -- (141.0791,305.8242) --
      (138.3164,315.2968) -- (135.1836,323.2032) -- (129.0570,339.2021) --
      (126.6139,345.9042);
    \path[fill] (123.5317,354.3600) -- (133.4399,343.6028) -- (126.9992,344.8473) --
      (122.8702,339.7499) -- cycle;
    \path[fill,line width=0.800pt] (146,291.2852) node[above right] (text62) {$1$};
    \path[draw] (131.9032,224.7135) -- (209.9916,224.7135);
    \path[fill] (218.9916,224.7135) -- (205.4916,219.0885) -- (208.8666,224.7135) --
      (205.4916,230.3385) -- cycle;
    \path[fill,line width=0.800pt] (169.8459,218.1823) node[above right] (text68)
      {$m$};
    \path[draw] (248.9916,224.7135) -- (348.3973,224.7135);
    \path[fill] (357.3973,224.7135) -- (343.8973,219.0885) -- (347.2723,224.7135) --
      (343.8973,230.3385) -- cycle;
    \path[fill,line width=0.800pt] (299.4005,218.1823) node[above right] (text74)
      {$1$};
    \path[draw] (361.9680,235.4945) -- (250.6785,350.5374);
    \path[fill] (244.4209,357.0061) -- (257.8501,351.2141) -- (251.4607,349.7289) --
      (249.7644,343.3922) -- cycle;
    \path[fill,line width=0.800pt] (299.4005,283.7297) node[above right] (text80)
      {$1$};
    \path[draw] (233.9916,352.7871) -- (233.9916,312.2503);
    \path[fill] (233.9916,303.2503) -- (228.3666,316.7503) -- (233.9916,313.3753) --
      (239.6166,316.7503) -- cycle;
    \path[fill,line width=0.800pt] (218.7885,332.5538) node[above right] (text86)
      {$m$};
    \path[draw] (233.9916,273.2503) -- (233.9916,248.7135);
    \path[fill] (233.9916,239.7135) -- (228.3666,253.2135) -- (233.9916,249.8385) --
      (239.6166,253.2135) -- cycle;
    \path[fill,line width=0.800pt] (218.7885,261.0171) node[above right] (text92)
      {$m$};
    \path[draw] (-14.1508,342.5950) -- (-3.1587,352.0944);
    \path[fill] (3.6508,357.9791) -- (-2.8855,344.8960) -- (-4.0099,351.3588) --
      (-10.2414,353.4079) -- cycle;
  \end{scope}

\end{tikzpicture}}
\caption{Automaton for the eight forbidden words}\label{f72}
\end{figure}

We claim that the block $(1mm1)(m11mm1)$ is also forbidden. 

Assume on the contrary that a sequence $(c_i)\in V'_{q,m}$ contains such a block.
Then we infer from the above list that
\begin{equation*}
(c_{n-1+i})\in (1mm1)(m11mm1)\set{1mm1,m11mm1}^{\infty}
\end{equation*}
for some $n\ge 1$.
Therefore $c_n=1$ and
\begin{equation*}
(c_{n+i})\ge (mm1)(m11mm1)(1mm1)^{\infty}=(mm1)(m11m)^{\infty}.
\end{equation*}
Applying Lemma \ref{l41} and the hypothesis $q\le r(m)$ we get
\begin{equation*}
\pi_q(c_{n+i})=\pi_q\left(mm1(m11m)^{\infty}\right)\ge \pi_{r_m}\left(mm1(m11m)^{\infty}\right)=m-1,
\end{equation*}
contradicting the first condition of Lemma \ref{l42}.
This contradiction proves our claim.

Since $(1mm1)(m11mm1)$ is forbidden, the above list corresponding to seven forbidden words may be further reduced: for $q\in (1,r(m)]$ each sequence $(c_i)\in V'_{q,m}$ belongs to the following list:
\begin{align*}
&m^{\infty},\\ 
&m^*(1m)^{\infty},\\ 
&m^*(1m)^*1(1mm1)^{\infty},\\
&m^*(1m)^*(1mm1)^{\infty}.
\end{align*}
(See also the corresponding automaton in Figure \ref{f72}.)
Hence $V'_{q,m}$ is countable. 

As in the previous proofs, the conclusion remains valid for all $1<q\le r(m)$.
\medskip 

\emph{Second step.}
We show that $V'_{q,m}$ has the power of continuum if $q>r(m)$. 
By Lemma \ref{l71} there exists a sufficiently large integer $k$ such that 
\begin{equation*}
\pi_q\left(mm(1m(11mm)^k)^{\infty}\right)< m-1.
\end{equation*}
We complete the proof by showing that 
\begin{equation*}
\set{1m(11mm)^k,1m(11mm)^{k+1}}^{\infty}\subset V'_{q,m}.
\end{equation*}

It suffices to check that if $(c_i)$ is a sequence in the left hand side set and $c_n=1$, then the conditions of Lemma \ref{l42} are satisfied.
Since $c_n=1$, we have 
\begin{equation*}
(1mm1)^{\infty}<(c_{n+i})\le mm(1m(11mm)^k)^{\infty}.
\end{equation*}
Applying Lemma \ref{l41} the conditions of Lemma \ref{l42} follow:
\begin{equation*}
\pi_q(c_{n+i})\le \pi_q\left(mm(1m(11mm)^k)^{\infty}\right)< m-1
\end{equation*}
and 
\begin{equation*}
\pi_q(c_{n+i})> \pi_q\left((1mm1)^{\infty}\right)>\frac{m}{q-1}-1.
\end{equation*}
\medskip 

We conclude from the preceding steps that $r_m=r(m)$. 
\end{proof}

\section{Proof of Proposition \ref{p15} (iii)}\label{s8}

As in the preceding two sections, we have $d=(10)^{\infty}$, $m_d\approx 2.80194$ and $M_d\approx 4.54646$.
We start by defining $m_4$:

\begin{lemma}\label{l81}\mbox{}

\begin{enumerate}[\upshape (i)]
\item The system of equations 
\begin{equation*}
\pi_q(m(m1)^{\infty})=m-1\qtq{and} q=m-1
\end{equation*}
has a unique solution $(m_4,q_4)$ with $q_4>1$.
We have explicitly 
\begin{equation*}
m_4=\frac{3+\sqrt{13}}{2}\approx 3.30278\qtq{and} q_4=\frac{1+\sqrt{13}}{2}\approx 2.30278.
\end{equation*}

\item For each $m\in [m_4,M_d]$ there exists a number $r(m)>1$ such that 
\begin{align*}
1<q<r(m)&\Longrightarrow \pi_q(m(m1)^{\infty})\ge m-1\qtq{and} q\le m-1,\\
q>r(m)&\Longrightarrow \pi_q(m(m1)^{\infty})<m-1.
\end{align*}

\item We have
\begin{equation}\label{81}
P_m\le r(m)<R_m,
\end{equation}
and the equality $r(m)=P_m$ holds only if $m=M_d$.
\end{enumerate}
\end{lemma}
\noindent See Figure \ref{f15}.

\begin{proof}\mbox{}

\begin{enumerate}[\upshape (i)]
\item See \eqref{916} below.

\item The function $q\mapsto \pi_q(m(m1)^{\infty})$ is continuous and decreasing in $q\in (1,\infty)$. 
Since 
\begin{equation*}
\lim_{q\to 1}\pi_q(m(m1)^{\infty})=\infty
\qtq{and}
\lim_{q\to \infty}\pi_q(m(m1)^{\infty})=0,
\end{equation*}
there exists a  number $r(m)>1$ satisfying the relation
\begin{equation*}
\pi_q(m(m1)^{\infty})-(m-1)\sim r(m)-q.
\end{equation*}

It remains to show that $r(m)\le m-1$.
For $m=m_4$ we have equality by (i).
For $m>m_4$ we have (see \eqref{916})
\begin{equation*}
\pi_{m-1}(m(m1)^{\infty})<m-1\qtq{and} \pi_{r(m)}(m(m1)^{\infty})=m-1,
\end{equation*}
so that $r(m)<m-1$ by the decreasingness of the function $q\mapsto \pi_q(m(m1)^{\infty})$.

\item This follows from the definition of $r(m)$, \eqref{917} and \eqref{918}.
\end{enumerate}
\end{proof}

\begin{proof}[Proof of Proposition \ref{p15} (iii)]\mbox{}

\emph{First step.} If $1<q\le r(m)$, then $V'_{m,q}$ is countable. 

Indeed, it follows from our assumptions that
\begin{equation*}
\pi_q(m(m1)^{\infty})\ge m-1 
\qtq{and}
q\le m-1.
\end{equation*}
We claim that the blocks $1mm$ and $11$ are forbidden.
As usual, we may assume that $2<q\le r(m)$.

If $(c_i)\in\set{1,m}^{\infty}$ contains a block $c_nc_{n+1}=11$, then 
\begin{equation*}
\pi_q(m-(c_{n+i}))\ge \frac{m-1}{q}\ge 1
\end{equation*}
and hence $(c_i)\notin V'_q$ by Lemma \ref{l42}.

If $(c_i)\in\set{1,m}^{\infty}$ does not contain any block $11$, but contains a block
\begin{equation*}
c_nc_{n+1}c_{n+2}=1mm,
\end{equation*}
then
\begin{equation*}
\pi_q(c_{n+i})\ge \pi_q\left( m(m1)^{\infty}\right)\ge m-1
\end{equation*}
by our assumptions, and hence $(c_i)\notin V'_q$ again by Lemma \ref{l42}.

It follows that $V'_{m,q}\subset\set{m^{\infty}, m^*(1m)^{\infty}}$, and hence it is countable. 
\medskip 

\emph{Second step.} If $q>r(m)$, then $V'_{m,q}$ has the power of continuum. 

Indeed, since $\pi_q(m(m1)^{\infty})< m-1$ by our assumption, we may fix a large integer $k$ satisfying 
\begin{equation*}
\pi_q\left( \left( m(m1)^k\right)^{\infty}\right)<m-1.
\end{equation*}
We claim that 
\begin{equation*}
\set{m(m1)^{\ell}\ :\ \ell\ge k}^{\infty}\subset V'_{m,q}.
\end{equation*}

For the proof we take an arbitrary sequence $(c_i)$ from the left hand side set. 
We have to check the conditions of Lemma \ref{l42}. 
If $c_n=1$, then 
\begin{equation*}
\pi_q(c_{n+i})\le \pi_q\left( \left( m(m1)^k\right)^{\infty}\right)<m-1
\end{equation*}
by the choice of $k$, and 
\begin{equation*}
\pi_q(m-(c_{n+i}))<\pi_q\left( m-(m1)^{\infty}\right)=\pi_q(m-\delta')<1
\end{equation*}
because $q>p_m$ and therefore $\delta'\in V'_{m,q}$ by property (vii) in Theorem \ref{t12}. 
\medskip 

It follows from the preceding steps that $r_m=r(m)$. 
\end{proof}

\section{Appendix}\label{s9}

We assume throughout this section that $m\ge 2$ and $q>1$.
As before, we will write $f\sim g$ if $f$ and $g$ have the same sign.

An elementary computation shows that the definitions
\begin{equation*}
P_m:=1+\sqrt{\frac{m}{m-1}}
\qtq{and}
R_m:=1+\frac{m}{m-1}
\end{equation*}
are equivalent to the relations 
\begin{equation}\label{91}
(m-1)P_m(P_m-2)=1
\qtq{and}
(m-1)(R_m-2)=1.
\end{equation}
They will be often used below without explicit reference.

The first two subsections are related to the proof of Proposition \ref{p14}, the following four subsections to that of Proposition \ref{p15}.
Accordingly, the notations $m_d$, $M_d$ refer to the corresponding  admissible sequences $0^{\infty}$ or $(10)^{\infty}$.

\subsection{The sequence $1^{\infty}$}\label{ss91}

We recall from \cite{KomLaiPed2010} that the first connected component of $[2,\infty)\setminus C$ is the interval $[2,M_d)$, associated with the smallest admissible sequence $d=0^{\infty}$, where $M_d=2.32472\ldots$ is the unique positive solution of the equation
\begin{equation*}
\pi_{P_m}(m^{\infty}-1^{\infty})=1,
\end{equation*}
and that $p_m=p_m''$ for each $m\in [2,m_d]$, where $p_m''$ is the unique positive solution of the equation
\begin{equation*}
\pi_{p_m''}(m^{\infty}-1^{\infty})=1.
\end{equation*}

It follows from the identity
\begin{equation*}
\pi_q(m^{\infty}-1^{\infty})-1=\frac{m-1}{q-1}-1=\frac{m-q}{q-1}
\end{equation*}
that
\begin{equation}\label{92}
p_m=m\qtq{for all} m\in [2,m_d]=[2,2.32472\ldots],
\end{equation}
and that
\begin{equation}\label{93}
\pi_{P_m}(m^{\infty}-1^{\infty})-1\sim m-P_m,
\end{equation}
whence $M_d=2.32472\ldots$ is the unique solution of the equation $P_m=m$.

Since 
\begin{align*}
m=P_m
&\Longleftrightarrow m-1=\sqrt{\frac{m}{m-1}}\\
&\Longleftrightarrow (m-1)^3=m\\
&\Longleftrightarrow (m-1)^3-(m-1)-1=0,
\end{align*}
we conclude that $M_d=1+\alpha$ where $\alpha=1.32472\ldots$ denotes the unique positive root of the polynomial $x^3-x-1$ (this is the first Pisot number).

\subsection{The sequence $m1^{\infty}$}\label{ss92}
We have
\begin{align*}
\pi_q(m1^{\infty})-(m-1)
&=\frac{m-1}{q}+\frac{1}{q-1}-(m-1)\\
&=\frac{q-(m-1)(q-1)^2}{q(q-1)}.
\end{align*}
Hence 
\begin{equation}\label{94}
\pi_q(m1^{\infty})=(m-1)\Longleftrightarrow q=r(m):=\frac{2m-1+\sqrt{4m-3}}{2m-2}.
\end{equation}

Furthermore, using \eqref{91} we obtain
\begin{align*}
\pi_{R_m}(m1^{\infty})-(m-1)
&\sim R_m-(m-1)(R_m-1)^2\\
&\sim R_m(R_m-2)-(R_m-1)^2\\
&=-1
\end{align*}
and 
\begin{align*}
\pi_{P_m}(m1^{\infty})-(m-1)
&\sim P_m-(m-1)(P_m-1)^2\\
&\sim P_m^2(P_m-2)-(P_m-1)^2\\
&=P_m^3-3P_m^2+2P_m-1\\
&=(P_m-1)^3-(P_m-1)-1.
\end{align*}
Recalling that $x^3-x-1=0$ for $x=\alpha$ and observing that $x^3-x-1<0$ for $x\in (0,\alpha)$, we conclude from the last two relations that 
\begin{equation}\label{95}
\pi_{R_m}(m1^{\infty})<(m-1)\qtq{for all} m\in [2,\infty)
\end{equation} 
and (since the function $m\mapsto P_m$ is decreasing)
\begin{equation}\label{96}
\pi_{P_m}(m1^{\infty})-(m-1)\sim 1+\alpha-m=M_d-m.
\end{equation} 

\subsection{The sequence $(m1)^{\infty}$}\label{ss93}

We recall from \cite{KomLaiPed2010} that the connected component $(m_d,M_d)=(2.80194\ldots,4.54646\ldots)$ of $[2,\infty)\setminus C$, associated with the  admissible sequence $d=\set{(10)^{\infty}}$, is defined by the equations 
\begin{equation*}
\pi_{P_{m_d}}((m_d1)^{\infty})=m-1\qtq{and} \pi_{P_{M_d}}(M_d^{\infty}-(M_d1)^{\infty})=1,
\end{equation*}
and that $p_m=\max\set{p_m',p_m''}$ for all $m\in [m_d,M_d]$, where $p_m'$ and $p_m''$ are uniquely defined by the equations 
\begin{equation*}
\pi_{p_m'}((m1)^{\infty})=m-1\qtq{and} \pi_{p_m''}(m^{\infty}-(m1)^{\infty})=1.
\end{equation*}

We have
\begin{align*}
\pi_q((m1)^{\infty})-(m-1)
&=\frac{1}{q-1}+(m-1)\pi_q((10)^{\infty})-(m-1)\\
&=\frac{1}{q-1}+\frac{(m-1)q}{q^2-1}-(m-1)\\
&=\frac{(q+1)-(m-1)(q^2-q-1)}{q^2-1}.
\end{align*}
Hence 
\begin{multline}\label{97}
\pi_q((m1)^{\infty})=(m-1)\Longleftrightarrow\\
q=r(m):=\frac{m-1+\sqrt{(m-1)^2+4}}{2}.
\end{multline}

Furthermore, using \eqref{91} we obtain
\begin{align*}
\pi_{P_m}((m1)^{\infty})-(m-1)
&\sim (P_m+1)-(m-1)(P_m^2-P_m-1)\\
&\sim (P_m+1)P_m(P_m-2)-(P_m^2-P_m-1).
\end{align*}
Hence
\begin{equation}\label{98}
\pi_{P_m}((m1)^{\infty})-(m-1)\sim P_m^3-2P_m^2-P_m+1
\end{equation}
after simplification, and, using the decreasingness of $m\mapsto P_m$,
\begin{equation}\label{99}
P_m^3-2P_m^2-P_m+1\sim m_d-m.
\end{equation} 

Finally, it follows from  the identity
\begin{align*}
\pi_q(m^{\infty}-(m1)^{\infty})-1
&=(m-1)\pi_q((01)^{\infty})-1\\
&=\frac{m-1}{q^2-1}-1
\end{align*}
that 
\begin{align*}
\pi_{P_m}(m^{\infty}-(m1)^{\infty})-1
&=\frac{m-1}{P_m^2-1}-1\\
&=\frac{1}{P_m(P_m-2)(P_m^2-1)}-1\\
&\sim 1-P_m(P_m-2)(P_m^2-1).
\end{align*}
Hence
\begin{equation}\label{910}
\pi_{P_m}(m^{\infty}-(m1)^{\infty})-1\sim -P_m^4+2P_m^3+P_m^2-2P_m+1
\end{equation}
after simplification, and, since the right hand side tends to the positive limit $1$ as $m\to \infty$ and $P_m\to 2$,
\begin{equation}\label{911}
-P_m^4+2P_m^3+P_m^2-2P_m+1\sim m-M_d.
\end{equation} 

\subsection{The sequence $(1m)^{\infty}$}\label{ss94}
We have
\begin{align*}
\pi_q(m^{\infty}-(1m)^{\infty})-1
&=(m-1)\pi_q((10)^{\infty})-1\\
&=\frac{(m-1)q}{q^2-1}-1\\
&=\frac{(m-1)q+(1-q^2)}{q^2-1}.
\end{align*}
Hence
\begin{multline}\label{912}
\pi_q(m^{\infty}-(1m)^{\infty})=1\Longleftrightarrow \\
q=r(m):=\frac{m-1+\sqrt{(m-1)^2+4}}{2}.
\end{multline} 

Furthermore, using \eqref{91},
\begin{align*}
\pi_{R_m}(m^{\infty}-(1m)^{\infty})-1
&\sim (m-1)R_m+(1-R_m^2)\\
&\sim R_m+(R_m-2)(1-R_m^2)\\
&=-R_m^3+2R_m^2+2R_m-2
\end{align*}
and 
\begin{align*}
\pi_{P_m}(m^{\infty}-(1m)^{\infty})-1
&\sim (m-1)P_m+(1-P_m^2)\\
&\sim 1+(P_m-2)(1-P_m^2)\\
&=-P_m^3+2P_m^2+P_m-1.
\end{align*}

If $m\in [2.80194\ldots,2.91286\ldots]$, then $R_m\in [2.5,3)$, and therefore 
\begin{align*}
-R_m^3+2R_m^2+2R_m-2
&=(R_m+1)\left( 1.25-(R_m-1.5)^2\right) -1\\
&\le \frac{R_m+1}{4}-1<0.
\end{align*}
We conclude that 
\begin{multline}\label{913}
\pi_{R_m}(m^{\infty}-(1m)^{\infty})<1\qtq{for all} \\
m\in [2.80194\ldots,2.91286\ldots].
\end{multline} 

Furthermore, using \eqref{99} we obtain that
\begin{equation}\label{914}
\pi_{P_m}(m^{\infty}-(1m)^{\infty})\ge 1\qtq{for all} m\ge 2.80194\ldots,
\end{equation} 
with equality only if $m=m_d=2.80194\ldots .$

\subsection{The sequence $mm1^{\infty}$}\label{ss95}
We have
\begin{align*}
\pi_q(mm1^{\infty})-(m-1)
&=\frac{m-1}{q}+\frac{m-1}{q^2}+\frac{1}{q-1}-(m-1)\\
&=\frac{q^2-(m-1)(q-1)(q^2-q-1)}{q^2(q-1)}\\
&=\frac{q^2-(m-1)(q^3-2q^2+1)}{q^2(q-1)},
\end{align*}
implying
\begin{equation}\label{915}
\pi_q(mm1^{\infty})-(m-1)\sim 1-(m-1)(q-2+q^{-2}).
\end{equation} 

\subsection{The sequence $m(m1)^{\infty}$}\label{ss96}
We have
\begin{align*}
\pi_q(m(m1)^{\infty})-(m-1)
&=\frac{m-1}{q}+\frac{1}{q-1}+\frac{m-1}{q^2-1}-(m-1)\\
&=\frac{q(q+1)-(m-1)(q^3-q^2-2q+1)}{q(q^2-1)}.
\end{align*}
We evaluate this expression for $q=m-1$, $R_m$ and $P_m$.

For $q=m-1$ we get 
\begin{align*}
\pi_{m-1}(m(m1)^{\infty})-(m-1)
&=\frac{q(q+1)-q(q^3-q^2-2q+1)}{q(q^2-1)}\\
&=\frac{-q(q^2-q-3)}{q^2-1}.
\end{align*}
Denoting by $q_4:=(1+\sqrt{13})/2$ the unique positive root of the numerator and setting $m_4:=1+q_4$, it follows that
\begin{equation}\label{916}
\pi_{m-1}(m(m1)^{\infty})-(m-1)\sim m_4-m.
\end{equation} 

Next we remark that
\begin{align*}
\pi_{R_m}(m(m1)^{\infty})&-(m-1)\\
&\sim R_m(R_m+1)-(m-1)(R_m^3-R_m^2-2R_m+1)\\
&\sim R_m(R_m+1)(R_m-2)-(R_m^3-R_m^2-2R_m+1)\\
&=-1,
\end{align*}
so that
\begin{equation}\label{917}
\pi_{R_m}(m(m1)^{\infty})<m-1.
\end{equation} 

Finally, 
\begin{align*}
\pi_{P_m}(m(m1)^{\infty})&-(m-1)\\
&\sim P_m(P_m+1)-(m-1)(P_m^3-P_m^2-2P_m+1)\\
&\sim P_m^2(P_m+1)(P_m-2)-(P_m^3-P_m^2-2P_m+1),
\end{align*}
yielding
\begin{equation*}
\pi_{P_m}(m(m1)^{\infty})-(m-1)\sim P_m^4-2P_m^3-P_m^2+2P_m-1.
\end{equation*} 
Using \eqref{911} we conclude that
\begin{equation}\label{918}
\pi_{P_m}(m(m1)^{\infty})-(m-1)\ge 0\qtq{if} m\in [2,M_d],
\end{equation} 
with equality only if $m=M_d$.

\end{document}